\documentclass{amsart}
\usepackage{amsmath,amsthm,amssymb,mathtools}
\usepackage[british]{babel}
\usepackage{enumitem}
\usepackage[latin1]{inputenc}
\usepackage{bussproofs}
\usepackage{url}
\usepackage{hyperref}

\newtheorem{theorem}{Theorem}
\numberwithin{theorem}{section}
\newtheorem{corollary}[theorem]{Corollary}
\newtheorem{lemma}[theorem]{Lemma}
\newtheorem{proposition}[theorem]{Proposition}
\theoremstyle{definition}

\newtheorem{question}[theorem]{Question}

\newcommand{\pa}{{\mathbf{PA}}}
\newcommand{\con}{\operatorname{Con}}
\newcommand{\rfn}{\operatorname{RFN}}
\newcommand{\feps}{F_{\varepsilon_0}}
\newcommand{\isigma}{\mathbf{I\Sigma}}

\begin{document}

\title{Short Proofs for Slow Consistency}

\author{Anton Freund}
\author{Fedor Pakhomov}


\begin{abstract}
Let $\operatorname{Con}(\mathbf T)\!\restriction\!x$ denote the finite consistency statement ``there are no proofs of contradiction in $\mathbf T$ with $\leq x$ symbols''. For a large class of natural theories $\mathbf T$, Pudl\'ak has shown that the lengths of the shortest proofs of $\operatorname{Con}(\mathbf T)\!\restriction\!n$ in the theory $\mathbf T$ itself are bounded by a polynomial in $n$. At the same time he conjectures that $\mathbf T$ does not have polynomial proofs of the finite consistency statements $\operatorname{Con}(\mathbf T+\operatorname{Con}(\mathbf T))\!\restriction\!n$. In contrast we show that Peano arithmetic~($\mathbf{PA}$) has polynomial proofs of \mbox{$\operatorname{Con}(\mathbf{PA}+\operatorname{Con}^*(\mathbf{PA}))\!\restriction\!n$}, where $\operatorname{Con}^*(\mathbf{PA})$ is the slow consistency statement for Peano arithmetic, introduced by S.-D.\ Friedman, Rathjen and Weiermann. We also obtain a new proof of the result that the usual consistency statement $\operatorname{Con}(\mathbf{PA})$ is equivalent to $\varepsilon_0$ iterations of slow consistency. Our argument is proof-theoretic, while previous investigations of slow consistency relied on non-standard models of arithmetic.
\end{abstract}

\maketitle

{\let\thefootnote\relax\footnotetext{\copyright~2020. This manuscript version is made available under the CC-BY-NC-ND 4.0 license \url{http://creativecommons.org/licenses/by-nc-nd/4.0/}. This is the accepted version of a paper published in the Notre Dame Journal of Formal Logic \textbf{61} (2020), no.~1, 31-49.}}

\noindent Consider the finite consistency statement $\con(\mathbf T)\!\restriction\!x$, which expresses that every proof in the theory~$\mathbf T$ that has $\leq x$ symbols is not a proof of contradiction. If $\mathbf T$ is consistent and $n$ is a fixed numeral then $\con(\mathbf T)\!\restriction\!n$ can be established in a very weak theory, simply by checking the finitely many proofs of length at most~$n$ explicitly ($\Sigma_1$-completeness). Note that the number of potential proofs to be checked is exponential in $n$.

As a refinement of G\"odel's second incompleteness theorem, H.~Friedman (unpublished) and Pudl\'ak~\cite{pudlak-length-consistency} have made estimates on the length of the shortest proofs of the statements $\con(\mathbf T)\!\restriction\!n$ in the theory $\mathbf T$ itself. For a large class of theories they have shown that no proof of $\con(\mathbf T)\!\restriction\!n$ in $\mathbf T$ can have less than $n^\varepsilon$ symbols, for a suitable $\varepsilon>0$ (see Section~\ref{sect:preliminaries} for a more detailed definition of proof length).

In the present paper we are more interested in upper bounds: Pudl\'ak~\cite{pudlak-length-consistency} has shown that for theories $\mathbf T$ from a wide class,
\begin{equation*}
 \mathbf T\vdash\con(\mathbf T)\!\restriction\!n\qquad\text{with polynomial in $n$ proofs}.
\end{equation*}
Here and in the sequel we adopt the following terminology: For a family of sentences $\varphi_n$ we say that $\mathbf{T}\vdash\varphi_n$ with polynomial in $n$ proofs if there is a polynomial $p(x)$ such that for each $n$ there is a $\mathbf{T}$-proof of $\varphi_n$ with at most $p(n)$ symbols.

At the same time Pudl\'ak conjectures that $\mathbf T$ does not have polynomial proofs for the finite consistency of considerably stronger theories. A survey of related conjectures and many new results can be found in~\cite{pudlak-incompleteness-finite-domain}. Specifically~\cite[Problem~I]{pudlak-length-consistency} asks whether we have
\begin{equation*}
 \mathbf T\vdash\con(\mathbf T+\con(\mathbf T))\!\restriction\!n\qquad\text{with polynomial in $n$ proofs?}
\end{equation*}
The expected negative answer remains open. This is no surprise because the negative answer would imply $\mathbf{co\mbox{\,\bf -}NP}\neq\mathbf{NP}$ and hence $\mathbf{P}\neq\mathbf{NP}$, as pointed out in \cite[Proposition~6.2]{pudlak-length-consistency}.

In contrast with this there is an interesting new result of Hrube\v{s} (see \cite[Theorem~3.7]{pudlak-incompleteness-finite-domain}) which tests the boundaries of Pudl\'ak's conjecture: For any formula $\varphi$ which is consistent with a sequential and finitely axiomatized theory $\mathbf T$ there is a true $\Pi_1$-formula $\psi$ such that we have $\mathbf T+\varphi\nvdash\psi$ but
\begin{equation*}
 \mathbf T\vdash\con(\mathbf T+\psi)\!\restriction\!n\qquad\text{with polynomial in $n$ proofs}.
\end{equation*}
Hrube\v{s} constructs the formula $\psi$ by Rosser-style diagonalization.

The aim of the present paper is to exhibit a natural example of the same phenomenon: We show
\begin{equation}\label{eq:short-proofs-slow-consistency}
 \pa\vdash\con(\pa+\con^*(\pa))\!\restriction\!n\qquad\text{with polynomial in $n$ proofs},
\end{equation}
where $\con^*(\pa)$ is the slow consistency statement due to S.-D.~Friedman, Rathjen and Weiermann~\cite{FRW-slow-consistency}.

This statement is defined as
\begin{equation*} \label{con*def}
 \con^*(\pa)\,\equiv\,\forall_x(\feps(x)\!\downarrow\,\rightarrow\con(\isigma_x)).
\end{equation*}
Here $\feps$ denotes the function at stage $\varepsilon_0$ of the fast-growing hierarchy. It is well-known that Peano arithmetic does not prove the statement $\forall_x\feps(x)\!\downarrow$, i.e.~it does not prove that the computation of $\feps(x)$ terminates for any input $x$ (see Section~\ref{sect:preliminaries} for more details). Moreover $\varepsilon_0$ is the first level of the fast-growing hierarchy with this property.  As shown in \cite{FRW-slow-consistency} the usual consistency statement $\con(\pa)$ is unprovable in $\pa+\con^*(\pa)$, while $\con^*(\pa)$ is still unprovable in~$\pa$. Note that the latter is an instance of Gödel's second incompleteness theorem since $\con^*(\pa)$ is equivalent (over $\isigma_1$) to the standard consistency statement for $\pa$ given by an alternative enumeration process of axioms, where the axiom of $\Sigma_n$-induction appears only at the step $\feps(n)$.

The main technical result of the present paper is Theorem~\ref{thm:finite-consistency-conditional}, which states
\begin{equation}\label{eq:main-technical-result}
\isigma_1\vdash\forall_x(\feps(x)\!\downarrow\,\rightarrow\con(\pa+\con^*(\pa))\!\restriction\!x).
\end{equation}
To deduce (\ref{eq:short-proofs-slow-consistency}) it suffices to provide polynomial in $n$ proofs of the statements $\feps(n)\!\downarrow$. The latter reduce to suitable instances of arithmetical transfinite induction, which can be proved by Gentzen's classical construction. In Section~\ref{sect:short-proofs} we give a polynomial version of this construction, inspired by the remarkable similarity between Gentzen's proof of transfinite induction and the method of shortening cuts.

Let us point out that we will in fact prove a stronger version of (\ref{eq:main-technical-result}), with the slow uniform $\Pi_2$-reflection principle $\rfn^*_{\Pi_2}(\pa)$ at the place of slow consistency. As a consequence, one can strengthen $\con^*(\pa)$ to $\rfn^*_{\Pi_2}(\pa)$ in statement (\ref{eq:short-proofs-slow-consistency}) as well. This is remarkable because the $\Pi_2$-statement $\rfn^*_{\Pi_2}(\pa)$ extends Peano arithmetic by a new provably total function (see \cite[Section~3]{freund-proof-length}). In contrast, the above result of Hrube\v{s} only adds $\Pi_1$-axioms, which do not yield any new provably total functions.

We note that the usual reflection principle $\rfn_{\Pi_2}(\pa)$ is $\isigma_1$-provably equivalent to $\forall x\,\feps(x)\!\downarrow$. There is a similar characterization for $\rfn^*_{\Pi_2}(\pa)$ in terms of the totality of a function $\feps^*$ that has intermediate rate of growth: It eventually dominates any function $F_{\alpha}$ with $\alpha<\varepsilon_0$, but is considerably slower than $\feps$ (see Section~\ref{sect:preliminaries} for more details).

 Proposition~\ref{prop:usual-consistency-false} shows that (\ref{eq:main-technical-result}) becomes false if one replaces $\con^*(\pa)$ by the usual consistency statement $\con(\pa)$. This means that Pudl\'ak's conjecture itself cannot be refuted by our approach.

Finally we remark that statement (\ref{eq:main-technical-result}) yields new proofs for several known results about slow consistency: In Corollary~\ref{cor:slow-doesnt-imply-usual} we deduce
\begin{equation*}
 \pa+\con^*(\pa)\nvdash\con(\pa),
\end{equation*}
as originally shown by S.-D.~Friedman, Rathjen and Weiermann~\cite{FRW-slow-consistency}.

Corollary~\ref{cor:eps-iterations-slow-consistency} recovers an even stronger result, first established in \cite{freund-slow-reflection,henk-pakhomov}: The usual consistency statement $\con(\pa)$ is equivalent to $\varepsilon_0$ iterations of slow consistency. The new proofs of these results are considerably different from those known so far: They only use proof-theoretic methods, whereas non-standard models of arithmetic were central to all previous arguments. Specifically, our proof of (\ref{eq:main-technical-result}) is based on Buchholz and Wainer's~\cite{buchholz-wainer-87} computational analysis of Peano arithmetic via an infinitary proof system. This is inspired by the appendix of~\cite{FRW-slow-consistency}, which contains a similar result to (\ref{eq:main-technical-result}), but without the occurrence of slow consistency. To formalize arguments about infinite proofs in the meta theory $\isigma_1$ we use the remarkably elegant notation systems of Buchholz~\cite{buchholz91}. More information on this approach can be found in Section~\ref{sect:proof-theory}.

\section{Preliminaries}\label{sect:preliminaries}

We begin this section with some general remarks about the length of proofs: Following Pudl\'ak~\cite{pudlak-length-consistency,pudlak-incompleteness-finite-domain} we count the total number of symbols in a proof, rather than just the number of inferences. In other words, the length of a proof is essentially the number of digits in its G\"odel code. We will be most interested in proofs in Peano arithmetic, which we take to be axiomatized as usual. Let us also agree to work with Hilbert-style proofs in sequence form, even if this choice does not appear to be essential: Kraj\'i\v{c}ek has given a polynomial simulation of sequence-proofs by tree-proofs while Eder has shown that Hilbert-style proofs and sequent calculus proofs with cuts are polynomially equivalent (see \cite{pudlak-length-98} for both results).

As is common in the literature on proof length we use the binary version of the numerals $\overline{n}$,
$$\overline{0}=0,\;\;\;\; \overline{2n+2}=\overline{n+1}(1+1),\;\;\;\; \overline{2n+1}=\overline{n}(1+1)+1.$$ It is easy to observe that the length of $\overline{n}$ is $\mathcal{O}(\lg(n))$. Thus the trivial lower bound for the length of a proof of the formula $\varphi(\overline n)$ is logarithmic rather than linear. As the present paper is only concerned with upper bounds that are polynomial, without any significant changes we could work with the usual unary numerals. If there is no danger of misunderstanding the $n$-th numeral will be denoted by $n$ rather than~$\overline n$.

Our next aim is to give a precise definition of slow consistency: Recall that the proof-theoretic ordinal of Peano arithmetic is the first fixed-point $\varepsilon_0=\omega^{\varepsilon_0}$ of ordinal exponentiation. To each limit ordinal $\alpha<\varepsilon_0$ one assigns a fundamental sequence $n\mapsto\{\alpha\}(n)$ with $\alpha=\sup_{n\in\omega}\{\alpha\}(n)$: Write $\alpha=\beta+\omega^\gamma$, where $0<\gamma<\alpha$ is the smallest exponent in the Cantor normal form of $\alpha$, and put
\begin{align*}
 \{\beta+\omega^\gamma\}(n)&=\beta+\omega^{\{\gamma\}(n)}&&\text{if $\gamma$ is a limit},\\
 \{\beta+\omega^\gamma\}(n)&=\beta+\omega^\delta\cdot (n+1)&&\text{if $\gamma=\delta+1$.}
\end{align*}
To extend the notion to $\varepsilon_0$ itself, consider the operation $(n,\alpha)\mapsto\omega^\alpha_n$ recursively defined by $\omega^\alpha_0=\alpha$ and $\omega^\alpha_{n+1}=\omega^{\omega^\alpha_n}$. Writing $\omega_n$ for $\omega^1_n$ we put
\begin{equation*}
 \{\varepsilon_0\}(n)=\omega_{n+1}.
\end{equation*}
It will later be convenient to have fundamental sequences for zero and successor ordinals: Here we put $\{0\}(n)=0$ and $\{\alpha+1\}(n)=\alpha$. Now we can define the fast-growing hierarchy of functions $F_\alpha:\mathbb N\rightarrow\mathbb N$ by recursion over $\alpha\leq\varepsilon_0$, setting
\begin{align*}
 F_0(n)&:=n+1,\\
 F_{\alpha+1}(n)&:=F_\alpha^{n+1}(n),\\
 F_\lambda(n)&:=F_{\{\lambda\}(n)}(n)\qquad\text{for $\lambda$ limit}.
\end{align*}
Note that the superscript in the second line denotes iterations.

It is well known that the ordinals below $\varepsilon_0$ can be arithmetized via the notion of Cantor normal form. We adopt the particularly efficient coding due to Sommer~\cite{sommer95}. This is convenient because Sommer gives a $\Delta_0$-formula $F_\alpha(x)=y$ with free variables $\alpha,x,y$ which uniformly defines the graphs of the functions in the fast-growing hierarchy. Basic properties of these functions (but not their totality) are then provable in $\isigma_1$. Strictly speaking, Sommer only considers ordinals $\alpha<\varepsilon_0$. The case $\alpha=\varepsilon_0$ has been treated in \cite[Section~2]{freund-proof-length}. We write $F_\alpha(x)\!\downarrow$ and~$F_\alpha\!\downarrow$ to abbreviate the formulas $\exists_y\,F_\alpha(x)=y$ resp.~$\forall_x\exists_y\,F_\alpha(x)=y$.

To complete the definition of slow consistency, recall that $\isigma_n$ is the fragment of Peano arithmetic which restricts induction to $\Sigma_n$-formulas. It is standard to define these fragments inside Peano arithmetic. Using the formalization  $\operatorname{Pr}_{\mathbf T}(\varphi)$ of provability in r.e.~theories $\mathbf{T}$ (see Feferman's paper \cite{feferman-arithmetization}) in the usual way we obtain a formula
\begin{equation*}
 \con(\isigma_x)\,\equiv\,\neg\operatorname{Pr}_{\isigma_x}(0=1)
\end{equation*}
with free variable $x$, which states that the $x$-th fragment is consistent. The defining equivalence of slow consistency, namely
\begin{equation*}
 \con^*(\pa)\,\equiv\,\forall_x(\feps(x)\!\downarrow\,\rightarrow\con(\isigma_x)),
\end{equation*}
is now fully explained. As we already mentioned in the introduction one could equivalently define the notion of slow consistency as the standard formalization of the notion of consistency for an alternative (slowed-down) enumeration of the axioms of $\pa$.

To avoid confusion we point out that \cite{freund-slow-reflection,henk-pakhomov} also considered different ``shifted'' versions of slow consistency. Namely, statements of the form
$$\forall_x(\feps(x)\!\downarrow\,\rightarrow\con(\isigma_{\max(x+s,1)})),$$ for some $s\in \mathbb{Z}$. We conjecture that our results could be generalized to cases of $s\le 1$. Nevertheless, for the sake of simplicity, we prefer to work with \mbox{S.-D.~Friedman}, Rathjen and Weiermann's~\cite{FRW-slow-consistency} original version $\con^*(\pa)$ in the present paper.  Note that since the analogue of Corollary~\ref{cor:eps-iterations-slow-consistency}  doesn't hold for $s>1$ (see \cite{henk-pakhomov,freund-slow-reflection}), our results cannot be generalized to these cases.

As in \cite{freund-slow-reflection} our analysis of slow consistency will make crucial use of the notion of slow uniform $\Pi_2$-reflection. Up to the aforementioned index shift this notion is defined as
\begin{equation*}
\rfn^*_{\Pi_2}(\pa)\,\equiv\,\forall_x(\feps(x)\!\downarrow\rightarrow\rfn_{\Pi_2}(\isigma_x)),
\end{equation*}
where $\rfn_{\Pi_2}(\isigma_x)$ refers to uniform $\Pi_2$-reflection (equivalently $\Sigma_1$-reflection) over the fragment~$\isigma_x$. Using a truth predicate $\operatorname{True}_{\Pi_2}(\cdot)$ for $\Pi_2$-formulas the latter can be given as a single formula, namely
\begin{equation*}
 \rfn_{\Pi_2}(\isigma_x)\,\equiv\,\forall_\varphi\,(\text{``$\varphi$ a $\Pi_2$-formula''}\land\operatorname{Pr}_{\isigma_x}(\varphi)\rightarrow\operatorname{True}_{\Pi_2}(\varphi)).
\end{equation*}
Note that we have
\begin{equation*}
 \isigma_1\vdash\rfn^*_{\Pi_2}(\pa)\rightarrow\con^*(\pa),
\end{equation*}
since $\rfn_{\Pi_2}(\isigma_x)$ implies $\con(\isigma_x)$ for each $x$.  In fact, slow $\Pi_2$-reflection implies that iterations of slow consistency are progressive (i.e., it provides the induction step to establish iterations of slow consistency). As for slow consistency one could show that uniform slow reflection is equivalent to uniform reflection for a slowed-down axiomatization of $\pa$. It is also interesting to observe that slow consistency and slow reflection can be derived from a common notion of slow proof (see e.g.~\cite[Section~2]{freund-proof-length}).

To conclude this section, let us recall the computational analysis of slow reflection from \cite[Section~3]{freund-proof-length}: It is well-known that $\rfn_{\Pi_2}(\isigma_x)$ is $\isigma_1$-provably equivalent to the totality of the function $F_{\omega_x}$ from the fast-growing hierarchy (see the model-theoretic argument in \cite{paris80}, the proof-theoretic argument in \cite{beklemishev03}, and the explicitly uniform proof in \cite{freund-characterize-reflection}). As in \cite[Corollary~3.2]{freund-proof-length} we thus have
\begin{equation*}
\isigma_1\vdash\rfn^*_{\Pi_2}(\pa)\leftrightarrow\forall_x(\feps(x)\!\downarrow\,\rightarrow F_{\omega_x}\!\downarrow).
\end{equation*}
Now consider the function
\begin{equation*}
\feps^*(x):=F_{\omega_{y}}(x)\quad\text{with}\quad y:=\feps^{-1}(x):=\max(\{z\leq x\,|\,\feps(z)\leq x\}\cup\{0\}).
\end{equation*}
According to \cite[Definition~3.3]{freund-proof-length} the graphs of $\feps^{-1}$ and $\feps^*$ are $\Delta_0$-definable. The function $\feps^{-1}$ is a monotone, non-stabilizing, and extemely slow growing function since it is an inverse of the total fast-growing function  $\feps$. It is easy to see that $\feps^{-1}$ is $\isigma_1$-provably total, although the fact that the values of $\feps^{-1}$ don't stabilize isn't provable even in $\pa$. The function $\feps^*$ dominates any provably total function of Peano arithmetic. But at the same time it is much slower than the usual function~$\feps$: The latter dominates any provably total function of the theory $\pa+\feps^*\!\downarrow$, as shown in~\cite[Theorem~3.10]{freund-proof-length}. Following \cite[Proposition~3.4]{freund-proof-length} the above equivalence can be transformed into
\begin{equation*}
\isigma_1\vdash\rfn^*_{\Pi_2}(\pa)\leftrightarrow\feps^*\!\downarrow.
\end{equation*}
It is the statement $\feps^*\!\downarrow$ (rather than $\rfn^*_{\Pi_2}(\pa)$ or $\con^*(\pa)$) that will be susceptive to proof-theoretic analysis.

\section{A Proof-Theoretic Approach to Slow Consistency}\label{sect:proof-theory}

In this section we analyse the theory $\pa+\feps^*\!\downarrow$ (see the previous section) by proof-theoretic means. Our approach is inspired by the appendix of~\cite{FRW-slow-consistency}, which is concerned with the finite consistency of $\pa$ rather than $\pa+\feps^*\!\downarrow$. As it turns out, the function $\feps^*$ is sufficiently slower than $\feps$ to be accomodated. As in \cite{FRW-slow-consistency} we use the infinitary proof system introduced by Buchholz and Wainer \cite{buchholz-wainer-87} in their analysis of the provably total functions of Peano arithmetic. As usual in ordinal analysis this system deduces sequents, i.e.~finite sets $\Gamma=\{\varphi_1,\dots,\varphi_n\}$ of formulas which are interpreted as the disjunction $\bigvee\Gamma\equiv\varphi_1\lor\dots\lor\varphi_n$. In particular, the empty sequent $\langle\rangle$ is the canonical way to express a contradiction. We write $\Gamma,\varphi$ rather than $\Gamma\cup\{\varphi\}$ for the extension of a sequent by a formula. A typical feature of ordinal analysis is the $\omega$-rule, which allows to deduce $\Gamma,\forall_x\varphi(x)$ if we have proofs of $\Gamma,\varphi(n)$ for each numeral. Note that this allows us to work with closed formulas only. In the presence of the $\omega$-rule, proofs are infinite trees. There are several ways to work with these trees in a finitistic meta theory. A particularly elegant approach is due to Buchholz~\cite{buchholz91}: The idea is to denote infinite proofs by finite terms which describe their role in the cut elimination process. For example, if $h$ denotes an infinite proof then we have another term $\mathcal E h$ which denotes the result of cut elimination. It turns out that all relevant operations on infinite proofs can be recast as primitive recursive functions on these finite terms. In particular, there is a primitive recursive function $\mathfrak s$ which computes the immediate subproofs. If the last rule of $h$ (also to be read off by a primitive recursive function) is not a cut then we have $\mathfrak s_n(\mathcal E h)=\mathcal E\mathfrak s_n(h)$. Intuitively this means that the $n$-th subproof of $\mathfrak E h$ is the result of cut elimination in the $n$-th subproof of $h$. Officially, the given equation is simply a clause in the definition of~$\mathfrak s$ by recursion over terms. In the same way one defines primitive recursive functions which read off the end-sequent, the ordinal height and the cut rank of (the infinite derivation denoted by) a given term. We say that $h$ codes an infinite proof $\vdash^\alpha_d\Gamma$ if it has ordinal height $\alpha$, cut rank $d$ and end-sequent $\Gamma$. By induction over terms one can show that proofs are ``locally correct''. For example, the theory $\isigma_1$ proves that the ordinal height of $\mathfrak s_n(h)$ is smaller than the ordinal height of $h$, provided~that the last rule of $h$ is not an axiom. It also proves that the end-sequents of the immediate subtrees $\mathfrak s_n(h)$ provide the assumptions required by the last rule of~$h$. In \cite{freund-characterize-reflection} we use Buchholz' approach to give a detailed formalization of the proof system from \cite{buchholz-wainer-87} in $\isigma_1$.

The proof system from \cite{buchholz-wainer-87} is special in that it keeps track of computational information about the derived sequents. For this purpose, the language of first order arithmetic is extended by a new predicate $\cdot\in N$. The proof system contains an axiom $0\in N$ and a rule which allows to infer $n+1\in N$ from $n\in N$. Thus we can deduce $n\in N$ for any numeral $n$, which suggests to interpret $N$ as the set of all natural numbers. The interpretation of $N$ becomes more interesting if we look at (essentially) cut free proofs of $\Sigma_1$-formulas: Now the idea is to interpret~$N$ as a finite approximation to the natural numbers, which bounds the size of existential witnesses. More precisely, a $\Sigma^N$-formula is built from arithmetical prime formulas and prime formulas $n\in N$ by the connectives $\land,\lor$ and ``relativized'' existential quantifiers $\exists_{y\in N}$ (as usual, we write $\exists_{y\in N}\varphi$ and $\forall_{y\in N}\varphi$ to abbreviate $\exists_y(y\in N\land\varphi)$ resp.~$\forall_y(y\notin N\lor\varphi)$). Such a formula is called true in~$K$ if it is true under the interpretation $N=\{m\,|\,3m<K\}$. As all quantifiers are relativized, truth in $K$ is primitive recursively decidable. By a $\Sigma^N$-sequent we mean a sequent that consist of $\Sigma^N$-formulas and prime formulas of the form $n\notin N$ (but the latter may not appear as subformulas of compound formulas). Such a sequent is called true in $K$ if it contains a $\Sigma^N$-formula which is true in $K$. Otherwise it is false in~$K$. The point of the formulas $n\notin N$ in a $\Sigma^N$-sequent $\Gamma$ is that they determine an ``input'', namely
\begin{equation*}
 k(\Gamma):=\max(\{2\}\cup\{3n\,|\,\text{the formula $n\notin N$ occurs in $\Gamma$}\}).
\end{equation*}
We say that the $\Sigma^N$-sequent $\Gamma$ is bounded by $F:\mathbb N\rightarrow\mathbb N$ if this function transforms the input $k(\Gamma)$ into a correct output, i.e.~if $\Gamma$ is true in $F(k(\Gamma))$. For example, the $\Sigma^N$-sequent $n\notin N,\exists_{y\in N}\,R(n,y)$ is bounded by $F$ if we have $R(n,m)$ for some number $m$ with $3m<F(\max\{2,3n\})$. In this terminology \cite[Lemma~5]{buchholz-wainer-87} reads as follows: If we have $\vdash^\alpha_0\Gamma$ where $\Gamma$ is a $\Sigma^N$-sequent then $\Gamma$ is bounded by the function $F_\alpha$ from the fast-growing hierarchy. To make this computational interpretation sound, one needs rather restrictive ordinal assignments. In particular, the ordinal ``height" assigned to a proof is related but not quite equal to the usual height of the underlying tree. From $\vdash^\alpha\Gamma$ and $\beta>\alpha$ we cannot in general infer $\vdash^\beta\Gamma$. This is only allowed if we have $\alpha<_k \beta$ for~$k=k(\Gamma)$. The latter means that one can descend from $\beta$ to $\alpha$ via $k$-th members of the fundamental sequence, i.e.~that there is a sequence $\beta=\gamma_0,\dots,\gamma_n=\alpha$ with $\gamma_{i+1}=\{\gamma_i\}(k)$. Note that the same relation is often denoted by $\beta\rightarrow_k\alpha$. The rule which allows to infer $\vdash^\beta\Gamma$ from $\vdash^\alpha\Gamma$ and $\alpha<_{k(\Gamma)}\beta$ is called accumulation. We refer to \cite[Section~3]{buchholz-wainer-87} for the full details of the infinite proof system. Also, the reader can find a complete list of rules as part of the case distinction in the proof of Proposition~\ref{prop:feps-cons-infinite} below.

To analyse the theory $\pa+\feps^*\!\downarrow$ we extend the proof system from \cite{buchholz-wainer-87} by the new axioms
\begin{equation*}
n\notin N,\exists_{y\in N}\feps^*(n)=y,
\end{equation*}
for any numeral $n$ and with arbitrary side formulas. More precisely, in the language of \cite{buchholz-wainer-87} the $\Delta_0$-formula $\feps^*(x)=y$ is represented by the corresponding elementary relation symbol. This ensures that $\exists_{y\in N}\feps^*(n)=y$ is a $\Sigma^N$-formula. We write $*\!\vdash^\alpha_d\Gamma$ to denote infinite proofs in the extended system. Starting with the new axioms, two applications of $\lor$-introduction yield $*\!\vdash^2_0 n\notin N\lor\exists_{y\in N}\feps^*(n)=y$. By the $\omega$-rule we infer $*\!\vdash^3_0\forall_{x\in N}\exists_{y\in N}\feps^*(x)=y$. In view of $3<_2\omega$ accumulation gives $*\!\vdash^\omega_0\forall_{x\in N}\exists_{y\in N}\feps^*(x)=y$, which means that the axiom $\feps^*\!\downarrow\,\equiv\forall_x\exists_y\feps^*(x)=y$ of the theory $\pa+\feps^*\!\downarrow$ can be embedded in the sense of \cite{buchholz-wainer-87}. In the presence of the new axioms, we cannot eliminate cuts with cut formula $\exists_{y\in N}\feps^*(n)=y$ resp.~$\forall_{y\in N}\feps^*(n)\neq y$. Technically, the simplest solution is to stipulate that the formula $\exists_{y\in N}\feps^*(n)=y$ and its negation have ``length'' $0$ (cf.~\cite[Definition~5]{buchholz-wainer-87}). Then the cut elimination result in \cite[Theorem~4]{buchholz-wainer-87} remains valid, i.e.~we can remove all cut formulas of length $>0$. Let us stress that a proof of cut rank $0$ may now contain cuts over formulas $n\in N$ and $\exists_{y\in N}\feps^*(n)=y$ (and their negations). As we will see, these cuts do not obstruct the computational interpretation sketched in the previous paragraph. We point out that the proof of \cite[Proposition~3.9]{freund-proof-length} embeds the axiom $\feps^*\!\downarrow$ in a different way: This allows for full cut elimination but it relies on the totality of~$\feps^*$, which is not available in a weak meta theory. We can now state the main technical result of this section:

\begin{proposition}\label{prop:feps-cons-infinite}
For each primitive recursive function $f$ there is an $n\in\mathbb N$ such that we have
\begin{equation*}
\isigma_1\vdash\forall_{x\geq n}(\feps(x+2)\!\downarrow\,\rightarrow*\!\nvdash^{\omega\cdot f(x)}_x\langle\rangle).
\end{equation*}
\end{proposition}

\begin{proof}
By meta induction on the definition of $f$ we get $\isigma_1\vdash\forall_{x>0}\, f(x)\leq F_n(x)$ for some $n>0$. To see that the theorem holds for this number, consider an arbitrary $x\geq n$ with $\feps(x+2)\!\downarrow$. Aiming at a contradiction, assume that we have~$*\!\vdash^{\omega\cdot f(x)}_x\langle\rangle$. By cut elimination (see \cite[Theorem~4]{buchholz-wainer-87}, as well as~\cite{buchholz91} for the formalization in $\isigma_1$) we get~$*\!\vdash^{\omega_x^{\omega\cdot f(x)}}_0\langle\rangle$. Recall that this proof may still contain cuts over prime formulas and formulas~$\exists_{y\in N}\,\feps^*(n)=y$. By primitive recursion we will define a sequence of pairs $\langle h_i,K_i\rangle$ with the following properties:
\begin{enumerate}
 \item[(i)] The number $h_i$ codes an infinite proof $*\!\vdash^{\alpha_i}_0\Gamma_i$, where $\Gamma_i$ is a $\Sigma^N$-sequent.
 \item[(iI)] We have $F_{\omega_{x+2}+\alpha_i}(k(\Gamma_i))=K_i$, and $\Gamma_i$ is false in $K_i$.
 \item[(iii)] We have $\alpha_{i+1}<_{k(\Gamma_i)}\alpha_i$ and $K_{i+1}\leq K_i$. If $K_{i+1}=K_i$ then $\Gamma_{i+1}=\Gamma_i$.
\end{enumerate}
Given such a sequence, the desired contradiction can be deduced as in the appendix of \cite{FRW-slow-consistency}: The inequality $K_{i+1}<K_i$ can only hold finitely often, say, only for~$i<M$ (to see this, assume the contrary and argue by $\Sigma_1$-induction to deduce a contradiction). By condition (iii) we get $\Gamma_M=\Gamma_{M+1}=\dots$ and thus
\begin{equation*}
 \alpha_M>_{k(\Gamma_M)}\alpha_{M+1}>_{k(\Gamma_M)}\dots
\end{equation*}
This is impossible, since $\isigma_1$ proves that any descending sequence of ordinals below $\varepsilon_0$ (and in fact some way above) with constant step-down must terminate (see \cite[Proposition~2.9]{ketonen-solovay} or the more explicit \cite[Lemma~4.7]{freund-characterize-reflection}). It remains to construct a primitive recursive sequence $\langle h_i,K_i\rangle$ which satisfies conditions~(i-iii). For $h_0$ we take the proof $*\!\vdash^{\omega_x^{\omega\cdot f(x)}}_0\langle\rangle$ from above. Before we can set $K_0:=F_{\omega_{x+2}+\omega_x^{\omega\cdot f(x)}}(2)$ we must show that this value is defined. By \cite[Section~2]{FRW-slow-consistency} we have
\begin{equation*}
 \omega_{x+3}=\omega^{\omega_{x+2}}>_{x+2}\omega^{\omega_{x+1}+1}>_{x+2}\omega^{\omega_{x+1}}\cdot 3>_{x+2}\omega_{x+2}\cdot 2+1.
\end{equation*}
By assumption $\feps(x+2)=F_{\omega_{x+3}}(x+2)$ is defined. Using \cite[Lemma~2.3]{FRW-slow-consistency} we conclude that $F_{\omega_{x+2}\cdot 2+1}(x+2)$ is defined as well. In view of \cite[Theorem~5.3]{sommer95} this is equivalent to saying that $F^{x+3}_{\omega_{x+2}\cdot 2}(x+2)$ is defined. In particular $F^2_{\omega_{x+2}\cdot 2}(x+2)$ is defined. By \cite[Proposition~5.4]{sommer95} we have $F_n(x+2)\leq F_{\omega_{x+2}\cdot 2}(x+2)$, so that $F_{\omega_{x+2}\cdot 2}(F_n(x+2))$ is defined as well. For later use, let us record
\begin{equation*}
 F_{\omega_{x+2}\cdot 2}(F_n(x+2))\leq F^2_{\omega_{x+2}\cdot 2}(x+2)\leq F^{x+3}_{\omega_{x+2}\cdot 2}(x+2)\leq\feps(x+2).
\end{equation*}
From $f(x)\leq F_n(x)\leq F_n(x+2)$ we infer $\omega^\omega>_{F_n(x+2)}\omega\cdot f(x)$, which in turn implies $\omega_{x+2}\cdot 2>_{F_n(x+2)}\omega_{x+2}+\omega_x^{\omega\cdot f(x)}$. Then \cite[Lemma~2.3]{FRW-slow-consistency} tells us that $F_{\omega_{x+2}+\omega_x^{\omega\cdot f(x)}}(F_n(x+2))$ is defined. In view of $2\leq F_n(x+2)$ we learn that the value $F_{\omega_{x+2}+\omega_x^{\omega\cdot f(x)}}(2)$ is defined, as desired. Let us also record
\begin{equation*}
 K_0=F_{\omega_{x+2}+\omega_x^{\omega\cdot f(x)}}(2)\leq F_{\omega_{x+2}\cdot 2}(F_n(x+2))\leq\feps(x+2).
\end{equation*}
Assuming that $\langle h_i,K_i\rangle$ is already constructed, we now define $\langle h_{i+1},K_{i+1}\rangle$ by case distinction on the last rule of the proof $h_i$ (which can be read off by a primitive recursive function, as explained above). Note that this cannot be the $\omega$-rule, as we assume that $\Gamma_i$ is a $\Sigma^N$-sequent.

\emph{Case.}  The last rule is an axiom of \cite{buchholz-wainer-87}. Then $\Gamma_i$ contains a true arithmetical prime formula, or the formula $0\in N$, or the formulas $m\notin N,m\in N$ for some numeral $m$. It is easy to see that each of these possibilities makes $\Gamma_i$ true in $K_i$, contradicting~(ii). Thus the last rule cannot, in fact, be an axiom of \cite{buchholz-wainer-87}.

\emph{Case.} The last rule is an axiom $m\notin N,\exists_{y\in N}\,\feps^*(m)=y$. Again, we show that this is impossible because it makes $\Gamma_i$ true in $K_i$: As $m\notin N$ occurs in $\Gamma_i$ we have
\begin{equation*}
 m\leq 3m\leq k(\Gamma_i)<K_i\leq K_0\leq\feps(x+2),
\end{equation*}
and thus $\feps^{-1}(m)<x+2$. Using \cite[Proposition~5.4]{sommer95} one infers
\begin{multline*}
 3\cdot\feps^*(m)=3\cdot F_{\omega_{\feps^{-1}(m)}}(m)\leq 3\cdot F_{\omega_{x+1}}(k(\Gamma_i))\leq F_{\omega_{x+1}}(k(\Gamma_i))^2<\\
<F_{\omega_{x+1}}^2(k(\Gamma_i))\leq F_{\omega_{x+1}+1}(k(\Gamma_i))\leq F_{\omega_{x+2}}(k(\Gamma_i))\leq F_{\omega_{x+2}+\alpha_i}(k(\Gamma_i))=K_i.
\end{multline*}
In particular all involved expressions are defined. This means that $\exists_{y\in N}\feps^*(m)=y$ is true in $K_i$, as promised.

\emph{Case.} Assume that $h_i$ ends with an (N)-rule (cf.~\cite{buchholz-wainer-87}). This means that $\alpha_i$ is of the form $\alpha_{i+1}+1$, that $\Gamma_i$ contains a formula $Sm\in N$, and that $h_i$ has an immediate subproof $*\,\vdash^{\alpha_{i+1}}_0\Gamma_i,m\in N$. More precisely, the sequent $\Gamma_i$ is of the form $\Gamma_i=\Gamma,Sm\in N$ and the immediate subproof has end-sequent $\Gamma,m\in N$. Here $\Gamma$ may or may not contain the formula $Sm\in N$. If it does not we may add this formula by weakening (see the remark after \cite[Definition~4]{buchholz-wainer-87}), to get the desired proof of $\Gamma,m\in N,Sm\in N=\Gamma_i,m\in N$. Let $h_{i+1}$ be (the code of) this proof $*\,\vdash^{\alpha_{i+1}}_0\Gamma_i,m\in N$. In particular we have $\Gamma_{i+1}=\Gamma_i,m\in N$. This implies $k(\Gamma_{i+1})=k(\Gamma_i)$ and thus
\begin{equation*}
 K_{i+1}:=F_{\omega_{x+2}+\alpha_{i+1}}(k(\Gamma_{i+1}))<F_{\omega_{x+2}+\alpha_i}(k(\Gamma_i))=K_i.
\end{equation*}
Note that $K_{i+1}$ can be computed by bounded minimization (i.e.\ as the smallest number below $K_i$ such that the elementary relation $K_{i+1}=F_{\omega_{x+2}+\alpha_{i+1}}(k(\Gamma_{i+1}))$ holds). It remains to show that $\Gamma_{i+1}$ is false in $K_{i+1}$. For all formulas in $\Gamma_i\subseteq\Gamma_{i+1}$ this holds by the obvious monotonicity property. Aiming at a contradiction, assume that $m\in N$ is true in~$K_{i+1}$, i.e.~that we have $3m<K_{i+1}$. Then we get
\begin{equation*}
 3(m+1)<K_{i+1}^2<F^2_{\omega_{x+2}+\alpha_{i+1}}(k(\Gamma_{i+1}))<F_{\omega_{x+2}+\alpha_i}(k(\Gamma_i))=K_i.
\end{equation*}
This makes $Sm\in N$ true in $K_i$, contradicting the assumption that $\Gamma_i$ is false in~$K_i$.

\emph{Case.} Assume that $h_i$ ends with a rule ($\land$). This means that $\alpha_i$ is of the form~$\alpha_{i+1}+1$, that $\Gamma_i$ contains a formula $\varphi_0\land\varphi_1$, and that $h_i$ has immediate subproofs $*\,\vdash^{\alpha_{i+1}}_0\Gamma_i,\varphi_j$ for $j=0,1$. By assumption $\varphi_0\land\varphi_1$ is false in $K_i$. As ``being false in $K_i$'' is primitive recursively decidable we can pick $j\in\{0,1\}$ such that $\varphi_j$ is false in $K_i$. Let $h_{i+1}$ be the corresponding subproof $*\,\vdash^{\alpha_{i+1}}_0\Gamma_i,\varphi_j$. As in the previous case we compute $K_{i+1}<K_i$. Monotonicity implies that $\Gamma_{i+1}$ is false in $K_{i+1}$.

\emph{Case.} Assume that $h_i$ ends with a rule ($\lor$). Similar to the previous case, let $h_{i+1}$ be the (only) immediate subproof of $h_i$.

\emph{Case.} Assume that $h_i$ ends with a rule ($\exists$). This means that $\alpha_i$ is of the form~$\alpha_{i+1}+1$, that $\Gamma_i$ contains a formula $\exists_{y\in N}\,\varphi(y)$, and that $h_i$ has an immediate subproof $*\,\vdash^{\alpha_{i+1}}_0\Gamma_i,m\in N\land\varphi(m)$ for some numeral $m$. As in the previous cases, let $h_{i+1}$ be that subproof.

\emph{Case.} Assume that $h_i$ ends with an accumulation rule. Then we have an immediate subproof $*\,\vdash^{\alpha_{i+1}}_0\Gamma_i$ with $\alpha_{i+1}<_{k(\Gamma_i)}\alpha_i$. Take $h_{i+1}$ to be this subproof, and observe $\Gamma_{i+1}=\Gamma_i$. In view of $\alpha_i\leq\alpha_0=\omega_x^{\omega\cdot f(x)}$ the ordinals $\omega_{x+2}$ and $\alpha_i$ mesh (cf.~\cite[Section~2]{FRW-slow-consistency}). This yields $\omega_{x+2}+\alpha_{i+1}<_{k(\Gamma_i)}\omega_{x+2}+\alpha_i$ and then
\begin{equation*}
 K_{i+1}=F_{\omega_{x+2}+\alpha_{i+1}}(k(\Gamma_{i+1}))\leq F_{\omega_{x+2}+\alpha_i}(k(\Gamma_i))=K_i,
\end{equation*}
as required for condition~(iii).

\emph{Case.} Assume that $h_i$ ends with a cut over the formula $m\in N$. This means that we have immediate subproofs $*\,\vdash^{\alpha_{i+1}}_0\Gamma_i,m\in N$ and $*\,\vdash^{\alpha_{i+1}}_0\Gamma_i,m\notin N$ with~$\alpha_i=\alpha_{i+1}+1$. If $m\in N$ is false in $F_{\omega_{x+2}+\alpha_{i+1}}(k(\Gamma_i))$ then we can take $h_{i+1}$ to be the first of these subproofs. Otherwise, let $h_{i+1}$ be the second subproof. Since $m\in N$ is true in $F_{\omega_{x+2}+\alpha_{i+1}}(k(\Gamma_i))$ we see
\begin{equation*}
 k(\Gamma_{i+1})=\max\{k(\Gamma_i),3m\}<F_{\omega_{x+2}+\alpha_{i+1}}(k(\Gamma_i)).
\end{equation*}
This implies
\begin{equation*}
 K_{i+1}=F_{\omega_{x+2}+\alpha_{i+1}}(k(\Gamma_{i+1}))\leq F^2_{\omega_{x+2}+\alpha_{i+1}}(k(\Gamma_i))<F_{\omega_{x+2}+\alpha_i}(k(\Gamma_i))=K_{i+1}.
\end{equation*}
Note that any $\Sigma^N$-formula in $\Gamma_{i+1}=\Gamma_i,m\notin N$ lies in $\Gamma_i$. By monotonicity, these formulas are false in $K_{i+1}$.

\emph{Case.} Assume that $h_i$ ends with a cut over an arithmetical prime formula (not involving the relation symbol $\cdot\in N$). Similar to the previous case, let $h_{i+1}$ be the subproof of the false premise.

\emph{Case.} Assume that $h_i$ ends with a cut over a formula $\exists_{y\in N}\,\feps^*(m)=y$. In fact, $\feps^*(m)=y$ may be any elementary relation $R(m,y)$. Then we have immediate subproofs $*\,\vdash^{\alpha_{i+1}}_0\Gamma_i,\exists_{y\in N}\,R(m,y)$ and $*\,\vdash^{\alpha_{i+1}}_0\Gamma_i,\forall_{y\in N}\,\neg R(m,y)$ with \mbox{$\alpha_i=\alpha_{i+1}+1$}. If $\exists_{y\in N}\,R(m,y)$ is false in $F_{\omega_{x+2}+\alpha_{i+1}}(k(\Gamma_i))$ then we can take $h_{i+1}$ to be the first of these subproofs. Otherwise $\exists_{y\in N}\,R(m,y)$ is true in $F_{\omega_{x+2}+\alpha_{i+1}}(k(\Gamma_i))$, i.e.~the relation $R(m,l)$ holds for some number $l\in\mathbb N$ with $3l<F_{\omega_{x+2}+\alpha_{i+1}}(k(\Gamma_i))$. By inversion \cite[Lemma~3]{buchholz-wainer-87} we get a proof
\begin{equation*}
 *\,\vdash^{\alpha_{i+1}}_0\Gamma_i,l\notin N,\neg R(m,l).
\end{equation*}
Let $h_{i+1}$ be (a code for) this proof. In Buchholz' \cite{buchholz91} formalization of infinite proofs it is immediate that $h_{i+1}$ can be computed by a primitive recursive function (see also the corresponding clause in \cite[Definition~3.5]{freund-characterize-reflection}). As in the previous case we have
\begin{equation*}
 k(\Gamma_{i+1})=\max\{k(\Gamma_i),3l\}<F_{\omega_{x+2}+\alpha_{i+1}}(k(\Gamma_i))
\end{equation*}
and thus
\begin{equation*}
 K_{i+1}=F_{\omega_{x+2}+\alpha_{i+1}}(k(\Gamma_{i+1}))\leq F^2_{\omega_{x+2}+\alpha_{i+1}}(k(\Gamma_i))<F_{\omega_{x+2}+\alpha_i}(k(\Gamma_i))=K_i.
\end{equation*}
Monotonicity ensures that $\Gamma_i$ is false in $K_{i+1}$. By the choice of $l$ the formula $\neg R(m,l)$ is false as well (independently of $K_{i+1}$).
\end{proof}

Switching back to the realm of finite proofs, we can infer the following:

\begin{theorem}
For any primitive recursive function $g:\mathbb N\rightarrow\mathbb N$ we have
\begin{equation*}
\isigma_1\vdash\forall_x(\feps(2x+5)\!\downarrow\rightarrow\con(\isigma_x+\feps^*\!\downarrow)\!\restriction\!g(x)).
\end{equation*}
\end{theorem}

Before we come to the proof, let us remark that the premise $\feps(2x+5)\!\downarrow$ can probably be weakened to $\feps(x)\!\downarrow$. To do so one would have to make the proof system from \cite{buchholz-wainer-87} more efficient, as in~\cite{rathjen-carnielli-91,wainer-fairtlough-98,freund-characterize-reflection}. Describing the necessary changes would make the presentation more difficult, and the improvement makes no difference for the existence of polynomial bounds. Also note that we obtain a nice bound in the next result.

\begin{proof}
We will construct a primitive recursive function $f$ such that $\isigma_1$ proves the following: Whenever $p$ is a derivation $\isigma_x+\feps^*\!\downarrow\,\vdash\, 0=1$ of length at most~$g(x)$, we have (a code for) an infinite derivation~$*\!\vdash^{\omega\cdot f(2x+3)}_{2x+3}\langle\rangle$. The previous proposition yields a number $n$ with
\begin{equation*}
\isigma_1\vdash\forall_{x\geq n}(\feps(2x+5)\!\downarrow\rightarrow *\!\nvdash^{\omega\cdot f(2x+3)}_{2x+3}\langle\rangle).
\end{equation*}
By the definition of $f$ we obtain
\begin{equation*}
\isigma_1\vdash\forall_{x\geq n}(\feps(2x+5)\!\downarrow\rightarrow\con(\isigma_x+\feps^*\!\downarrow)\!\restriction\!g(x)).
\end{equation*}
This is sufficient, as the finitely many cases $x<n$ are covered by $\Sigma_1$-completeness. The construction of $f$ takes place in several stages: First, translate $p$ into a proof $f_0(p)$ of $\langle\rangle$ in the finite sequent calculus of \cite[Section~2]{buchholz-wainer-87}. In particular $f_0$ has to eliminate function symbols, since \cite{buchholz-wainer-87} implements addition and multiplication as relations (this simplifies case (v) in the proof of \cite[Lemma~2]{buchholz-wainer-87}; a similar proof system with function symbols can be found in~\cite{wainer-fairtlough-98}). For $x\geq 1$ it is indeed possible to eliminate the function symbols for addition and multiplication, without increasing the logical complexity of the induction formulas. By choosing $n\geq 1$ above we do not have to worry about the case $x=0$. Next, transform $f_0(p)$ into a proof $f_1(p)$ in which all free cuts have been eliminated (see e.g.~\cite[Section~2.4.6]{buss-introduction-98}). This means that the only remaining cut formulas appear in some axiom. Considering the implementation of induction in \cite[Section~2]{buchholz-wainer-87} the most complex cuts in $f_1(p)$ are of the form
\begin{equation*}
 \exists_y(\varphi(y)\land\neg\varphi(Sy)),
\end{equation*}
where $\varphi$ is an induction formula of complexity $\Sigma_x$. Since the object language used in \cite{buchholz-wainer-87} contains a relation symbol for each elementary relation we may assume that $\varphi$ is a ``strict $\Sigma_x$-formula'' of the form
\begin{equation*}
 \varphi\equiv\exists_{z_1}\forall_{z_2}\dots\exists/\forall_{z_x}\theta,
\end{equation*}
where $\theta$ is a prime formula (cf.~\cite[Section~1]{freund-characterize-reflection}). Now embed $f_1(p)$ into the infinitary calculus of \cite[Section~3]{buchholz-wainer-87}, extended by the axioms \mbox{$n\notin N,\exists_{y\in N}\feps^*(n)=y$} described above. We have already seen that these new axioms yield an infinite proof $*\!\vdash^\omega_0\forall_{x\in N}\exists_{y\in N}\,\feps^*(x)=y$, which embeds the axiom~$\feps^*\!\downarrow$ in the sense of~\cite[Lemma~2]{buchholz-wainer-87}. The cuts in the finite proof $f_1(p)$ carry over to the infinitary system. More precisely, the cut over $\exists_y(\varphi(y)\land\neg\varphi(Sy))$ mentioned above is transformed into a cut over the formula
\begin{equation*}
 \exists_y(y\in N\land(\varphi^N(y)\land\neg\varphi^N(Sy))),
\end{equation*}
where we have
\begin{equation*}
 \varphi^N\equiv\exists_{z_1\in N}\forall_{z_2\in N}\dots\exists/\forall_{z_x\in N}\theta\equiv\exists_{z_1}(z_1\in N\land\forall_{z_2}(z_2\notin N\lor(\cdots\theta\cdots))).
\end{equation*}
This means that the cut formulas in the infinite proof may have length $2x+3$. The height of the embedded proof is of the form $\omega\cdot l$, where $l$ depends on the height of~$f_1(p)$, and on the height of formulas and terms in $f_1(p)$. Thus we get an infinite proof
\begin{equation*}
 *\,\vdash^{\omega\cdot f_2(p)}_{2x+3}\langle\rangle,
\end{equation*}
 where $f_2$ is a primitive recursive function. In Buchholz'~\cite{buchholz91} formalization it is trivial to compute codes for these embedded proofs, because the finite proof itself represents its infinite embedding. Now set
\begin{equation*}
 f(2x+3):=\max\{f_2(p)+1\,|\,\text{$p$ is a proof $\isigma_x+\feps^*\!\downarrow\,\vdash\, 0=1$ of length $\leq g(x)$}\}.
\end{equation*}
Clearly $f$ is primitive recursive. Whenever $p$ is a derivation $\isigma_x+\feps^*\!\downarrow\,\vdash\, 0=1$ of length at most $g(x)$ we get
\begin{equation*}
 *\,\vdash^{\omega\cdot f(2x+3)}_{2x+3}\langle\rangle
\end{equation*}
by accumulation, since $f_2(p)<f(2x+3)$ implies $\omega\cdot f_2(p)<_k\omega\cdot f(2x+3)$ for any number $k$.
\end{proof}

Concerning the finite consistency of slow reflection and slow consistency over~$\pa$ we get the following result. In the next section we will construct short proofs of $\feps(n)\!\downarrow$ in Peano arithmetic, to get short (unconditional) proofs of the finite consistency statements.

\begin{theorem}\label{thm:finite-consistency-conditional}
We have
\begin{equation*}
\isigma_1\vdash\forall_x(\feps(x)\!\downarrow\rightarrow\con(\pa+\con^*(\pa))\!\restriction\!x),
\end{equation*}
and the same with $\rfn^*_{\Pi_2}(\pa)$ at the place of $\con^*(\pa)$.
\end{theorem}
\begin{proof}
By \cite[Proposition~3.4]{freund-proof-length} we have
\begin{equation*}
 \isigma_1\vdash\rfn^*_{\Pi_2}(\pa)\leftrightarrow\feps^*\!\downarrow.
\end{equation*}
Just as for usual provability, slow reflection implies slow consistency, i.e.~we have
\begin{equation*}
  \isigma_1\vdash\rfn^*_{\Pi_2}(\pa)\rightarrow\con^*(\pa).
\end{equation*}
Now the claim follows from the previous theorem, observing that the length of a proof bounds the complexity of the induction formulas that it contains. More precisely, if a proof of length $x$ contains an induction axiom for an induction formula with~$y$ quantifiers then we have $2y+5<x$. In other words, we have $y\leq i(x)$ for
\begin{equation*}
i(x)=\min\{y\,|\,2(y+1)+5\geq x\}.
\end{equation*}
Let $g_0:\mathbb N^2\rightarrow\mathbb N$ be a primitive recursive function with the following property: If $p$ codes a proof $\pa+\con^*(\pa)\vdash 0=1$ of length at most $x$ then $g_0(p,x)$ codes a proof of $0=1$ in the theory $\isigma_{i(x)}+\con^*(\pa)$. As we have just seen, this can be achieved by prenexing and contracting quantifiers in the induction formulas. As the finitely many cases $x\leq 7$ are covered by $\Sigma_1$-completeness we may assume $x\geq 8$. Then we have $i(x)>0$, and the above implications allow us to transform $g_0(p,x)$ into a code $g_1(p,x)$ of a proof~$\isigma_{i(x)}+\feps^*\!\downarrow\,\vdash 0=1$. Writing $g_2(p,x)$ for the length of the proof $g_1(p,x)$, consider the primitive recursive function
\begin{multline*}
 g(y):=\max\{g_2(p,2(y+1)+5)\,|\,\text{$p$ is a proof $\pa+\con^*(\pa)\vdash 0=1$}\\
 \text{of length $\leq 2(y+1)+5$}\}.
\end{multline*}
Using the previous theorem for this function $g$ we can deduce the claim: Consider a number $x$ with $\feps(x)\!\downarrow$. As before we may assume $x\geq 5$, so that we have $2\cdot i(x)+5\leq x$ and thus $\feps(2\cdot i(x)+5)\!\downarrow$. Aiming at a contradiction, assume that $p$ is a proof $\pa+\con^*(\pa)\vdash 0=1$ of length at most $x\leq 2(i(x)+1)+5$. By construction $g_1(p,2(i(x)+1)+5)$ is a proof $\isigma_{i(x)}+\feps^*\!\downarrow\,\vdash 0=1$ of length at most~$g(i(x))$. This contradicts $\con(\isigma_{i(x)}+\feps^*\!\downarrow)\!\restriction\!g(i(x))$, as provided by the previous theorem.
\end{proof}

In particular, we can establish the following known result by purely proof-theoretic methods. All previous proofs used non-standard models of arithmetic.

\begin{corollary}[{\cite{FRW-slow-consistency}}]\label{cor:slow-doesnt-imply-usual}
 We have
\begin{equation*}
 \pa+\con^*(\pa)\nvdash\con(\pa).
\end{equation*}
\end{corollary}
\begin{proof}
 In view of G\"odel's second incompleteness theorem it suffices to show
\begin{equation*}
 \pa+\con(\pa)\vdash\con(\pa+\con^*(\pa)).
\end{equation*}
Internalizing the previous theorem we have
\begin{equation*}
 \isigma_1\vdash\operatorname{Pr}_\pa(\forall_x(\feps(x)\!\downarrow\rightarrow\con(\pa+\con^*(\pa))\!\restriction\!x)),
\end{equation*}
where $\operatorname{Pr}_\pa(\cdot)$ formalizes provability in Peano arithmetic. It is easy to see that we have $\isigma_1\vdash\forall_x\operatorname{Pr}_\pa(\feps(\dot x)\!\downarrow)$, as in \cite[Lemma~1.5]{freund-slow-reflection}. Together we get
\begin{equation*}
 \isigma_1\vdash\forall_x\operatorname{Pr}_\pa(\con(\pa+\con^*(\pa))\!\restriction\!\dot x).
\end{equation*}
Since consistency implies $\Pi_1$-reflection this yields
\begin{equation*}
 \isigma_1+\con(\pa)\vdash\forall_x\con(\pa+\con^*(\pa))\!\restriction\!x,
\end{equation*}
as desired.
\end{proof}

Beyond the previous corollary, we can bound the consistency strength of slow reflection. Note that the following is weaker than the known result, which accomodates slow reflection for arbitrary logical complexity, rather than just $\Pi_2$-reflection.

\begin{corollary}[{\cite{freund-slow-reflection}}]
We have
\begin{equation*}
\isigma_1\vdash\con(\pa)\rightarrow\con(\pa+\rfn^*_{\Pi_2}(\pa)).
\end{equation*}
\end{corollary}
\begin{proof}
 As in the previous proof, using Theorem~\ref{thm:finite-consistency-conditional} for slow reflection rather than slow consistency.
\end{proof}

As a consequence, we obtain a new proof of the following result. Transfinite iterations of slow consistency are defined in \cite[Section~3]{freund-slow-reflection} and \cite[Section~4]{henk-pakhomov}.

\begin{corollary}[{\cite{freund-slow-reflection,henk-pakhomov}}]\label{cor:eps-iterations-slow-consistency}
The usual consistency statement for Peano arithmetic is equivalent to $\varepsilon_0$ iterations of slow consistency, over the base theory $\isigma_1$.
\end{corollary}
\begin{proof}
As for the usual notion of consistency, slow $\Pi_2$-reflection implies that iterations of slow consistency are progressive. In other words, we have
\begin{equation*}
 \pa+\rfn^*_{\Pi_2}(\pa)\vdash\forall_\gamma(\forall_{\beta\prec\gamma}\con^*_\beta(\pa)\rightarrow\con^*_\gamma(\pa)),
\end{equation*}
where $\con^*_\gamma(\pa)$ denotes the $\gamma$-th iterate of slow consistency. Also recall Gentzen's result that Peano arithmetic proves transfinite induction up to any fixed ordinal below $\varepsilon_0$. Putting these observations together we get
\begin{equation*}
 \pa+\rfn^*_{\Pi_2}(\pa)\vdash\con^*_\alpha(\pa)
\end{equation*}
for each $\alpha\prec\varepsilon_0$. This fact itself can be proved in the theory $\isigma_1$. Now assume $\con(\pa)$ and invoke the previous corollary to get $\con(\pa+\rfn^*_{\Pi_2}(\pa))$. As consistency implies $\Pi_1$-reflection we learn that $\con^*_\alpha(\pa)$ does indeed hold for all ordinals~$\alpha\prec\varepsilon_0$. Thus we have proved the direction
\begin{equation*}
 \isigma_1\vdash\con(\pa)\rightarrow\forall_{\alpha\prec\varepsilon_0}\con^*_{\alpha}(\pa).
\end{equation*}
The converse direction follows from a result of Schmerl and Beklemishev, as in \cite[Corollary~3.6]{freund-slow-reflection} or \cite[Lemma~18]{henk-pakhomov}.
\end{proof}

\section{Proofs of Finite Consistency}\label{sect:short-proofs}

In this section we construct proofs $\pa\vdash\feps(n)\!\downarrow$ which are polynomial in~$n$. Together with Theorem~\ref{thm:finite-consistency-conditional} this will yield polynomial proofs of the finite consistency statements $\con(\pa+\con^*(\pa))\!\restriction\! n$. Note that there are ``naive" proofs of $\feps(n)\!\downarrow$ via $\Sigma_1$-completeness. However, as the value $\feps(n)$ is extremely large, these proofs are very long. Indeed, we know from \cite{freund-proof-length} that proofs of $\feps(n)\!\downarrow$ in the fragment $\isigma_n$ cannot be constructed by primitive recursion, let alone in a polynomial way. This means that we have to exhaust the full power of Peano arithmetic to get the desired polynomial proofs.

First we reduce $\feps(n)\!\downarrow$ to an appropriate instance of transfinite induction: Consider a formula $\psi\equiv\psi(\gamma)$, possibly with further variables. As in~\cite{feferman-reflections} we~write
\newcommand{\prog}{\operatorname{Prog}}
\newcommand{\ti}{\operatorname{TI}}
\begin{align*}
\prog(\widehat\gamma\,\psi)&\equiv\forall_\alpha(\forall_{\beta\prec\alpha}\psi(\beta)\rightarrow\psi(\alpha)),\\
\ti(\alpha,\widehat\gamma\,\psi)&\equiv\prog(\widehat\gamma\,\psi)\rightarrow\forall_{\beta\prec\alpha}\psi(\beta).
\end{align*}
It is standard to observe the following:

\begin{lemma}\label{lem:ti-to-feps}
We have
\begin{equation*}
\isigma_1\vdash\forall_x(\ti(\omega_{x+1},\widehat\gamma\,F_\gamma\!\downarrow)\rightarrow\feps(x)\!\downarrow).
\end{equation*}
\end{lemma}
\begin{proof}
Using \cite[Theorem~5.3]{sommer95} we show that $\forall_{\beta\prec\alpha}\,F_\beta\!\downarrow$ implies $F_\alpha\!\downarrow$: For~$\alpha=0$ we have $F_0(x)=x+1$ by definition. If $\alpha$ is a limit then $F_\alpha(x)\!\downarrow$ reduces to $F_{\{\alpha\}(x)}(x)\!\downarrow$, where $\{\alpha\}(x)\prec\alpha$ is the $x$-th member of the fundamental sequence. In case $\alpha=\beta+1$ we must establish $F_\beta^{x+1}(x)\!\downarrow$ for all $x$. Indeed, we show $F_\beta^i(x)\!\downarrow$ by induction on $i$. In the step we get $F_\beta(F^i_\beta(x))\!\downarrow$ by induction hypothesis and the assumption $\forall_{\beta\prec\alpha}\,F_\beta\!\downarrow$. Unfortunately, \cite[Theorem~5.3]{sommer95} only provides $F_\beta^{i+1}(x)=F^i_\beta(F_\beta(x))$ but not~$F_\beta^{i+1}(x)=F_\beta(F^i_\beta(x))$. The solution is to establish
\begin{equation*}
\forall_{j\leq i}\forall_{y\leq F^i_\beta(x)}(F^j_\beta(y)=F^i_\beta(x)\rightarrow F^{j+1}_\beta(y)=F_\beta(F^i_\beta(x)))
\end{equation*}
by induction on $j$. For $j=i$ and $y=x$ we get $F^{i+1}_\beta(x)=F_\beta(F^i_\beta(x))$, as required. We have thus established
\begin{equation*}
\isigma_1\vdash\prog(\widehat\gamma\,F_\gamma\!\downarrow).
\end{equation*}
Together with the assumption $\ti(\omega_{x+1},\widehat\gamma\,F_\gamma\!\downarrow)$ we obtain $\forall_{\beta\prec\omega_{x+1}}\,F_\beta\!\downarrow$. Another application of $\prog(\widehat\gamma\,F_\gamma\!\downarrow)$ yields $F_{\omega_{x+1}}\!\downarrow$. In particular we have $F_{\omega_{x+1}}(x)\!\downarrow$, which is equivalent to $\feps(x)\!\downarrow$ by definition.
\end{proof}

By a classical result of Gentzen, transfinite induction up to each ordinal $\omega_{n+1}$ can be proved in Peano arithmetic. We must show that this is possible with proofs of polynomial length. For this purpose it will be useful to work with schematic proofs: Add a unary predicate variable $R$ to the language of arithmetic and write $\pa[R]$ for the version of Peano arithmetic, where the induction scheme is extended to all arithmetic formulas with $R$. We fix some first-order variable $\gamma$. For formulas $\varphi[R]$ and $\psi[R](\gamma)$ of the extended language we write $\varphi[\psi[R]]$ for the formula that results from $\varphi[R]$, when one replaces all subformulas $Rt$ by the formula $\psi[R](t)$, renaming bound variables as necessary.

The crucial ingredient of Gentzen's proof of transfinite induction is the ``jump formula"
\begin{equation*}
J[R](\gamma)\equiv\forall_\beta(\forall_{\delta\prec\beta}\,R\delta\rightarrow\forall_{\delta\prec\beta+\omega^\gamma}\,R\delta).
\end{equation*}
It is well-known that this formula has the following property:

\begin{lemma}\label{lem:jump-transfinite-induction}
We have
\begin{equation*}
\pa[R]\vdash\forall_\alpha(\ti(\alpha,\widehat\gamma\,J[R])\rightarrow\ti(\omega^\alpha,\widehat\gamma\, R\gamma)).
\end{equation*}
\end{lemma}
\begin{proof}
The crucial observation is
\begin{equation*}
\pa[R]\vdash\prog(\widehat\gamma\, R\gamma)\rightarrow\prog(\widehat\gamma\,J[R]).
\end{equation*}
Details can, for example, be found in the proof of \cite[Lemma~4.4]{sommer95}. Aiming at $\ti(\omega^\alpha,\widehat\gamma\, R\gamma)$, let us now assume $\prog(\widehat\gamma\, R\gamma)$. Using the assumption $\ti(\alpha,\widehat\gamma\,J[R])$ we obtain $\forall_{\beta\prec\alpha}\,J[R](\beta)$, and then
\begin{equation*}
J[R](\alpha)\equiv\forall_\beta(\forall_{\delta\prec\beta}\,R\delta\rightarrow\forall_{\delta\prec\beta+\omega^\alpha}\,R\delta).
\end{equation*}
Setting $\beta=0$ the premise $\forall_{\delta\prec\beta}\,R\delta$ is trivial. Thus we get $\forall_{\delta\prec\omega^\alpha}\,R\delta$, as required for $\ti(\omega^\alpha,\widehat\gamma\, R\gamma)$.
\end{proof}

The ordinals $\omega_n$ can be reached by iterating the construction: Substitute $J[R]$ for the predicate variable $R$, to see that $\ti(\alpha,\widehat\gamma\,J[J[R]])$ implies $\ti(\omega^\alpha,\widehat\gamma\,J[R])$. With $\omega^\alpha$ at the place of $\alpha$ the lemma shows that $\ti(\omega^\alpha,\widehat\gamma\,J[R])$ implies $\ti(\omega_2^\alpha,\widehat\gamma\, R\gamma)$. Combining these implications we obtain
\begin{equation*}
 \pa[R]\vdash\ti(\alpha,\widehat\gamma\,J[J[R]])\rightarrow\ti(\omega_2^\alpha,\widehat\gamma\, R\gamma).
\end{equation*}
Starting with the trivial case $\alpha=1$ we get a proof of $\ti(\omega_2,\widehat\gamma\, R\gamma)$ outright. This does enable us to reach arbitrary ordinals $\omega_n$, but the length of the proof grows exponentially in $n$: As $J[R]$ contains two occurrences of $R$ the formula $J[J[R]]$ will contain four such occurrences, etc. It is well-known how to circumvent this problem: For convenience we assume that the biconditional is a primitive connective (this can be avoided as in~\cite[Section~3.2]{pudlak-length-98}). Then each formula is equivalent to a formula with a single occurrence of $R$, by~\cite[Chapter~7]{ferrante-rackoff-79}. In the present case it suffices to observe the equivalence between $R\delta_0\rightarrow R\delta_1$ and
\begin{multline*}
\theta(\delta_0,\delta_1)\equiv\exists_{v_0,v_1}[\forall_z(\bigvee_{i=0,1}z=\delta_i\rightarrow(Rz\leftrightarrow\bigwedge_{i=0,1}(z=\delta_i\rightarrow v_i=1)))\land{}\\
{}\land(v_0=1\rightarrow v_1=1)].
\end{multline*}
Thus the above jump formula $J[R]$ is equivalent to
\begin{equation*}
J'[R](\gamma)\equiv\forall_\beta\exists_{\delta_0\prec\beta}\forall_{\delta_1\prec\beta+\omega^\gamma}\theta(\delta_0,\delta_1).
\end{equation*}
Now put
\begin{align*}
J'_0[R](\gamma)&\equiv R\gamma,\\
J'_{n+1}[R](\gamma)&\equiv J'[J'_n[R]].
\end{align*}
As promised, these modified iterations satisfy the following:

\begin{lemma}
 The length of the formulas $J'_n[R]$ is bounded by a polynomial in $n$.
\end{lemma}
\begin{proof}
Assuming that all variables have length one (i.e.~neglecting the lengths of their indices) we get a linear bound: Suppose that $J'_n[R]$ has length at most $k\cdot n$, where $k$ bounds the length of $J'[R]$. As $J'[R]$ contains a single occurrence of $R$ the length of $J'_{n+1}[R]\equiv J'[J'_n[R]]$ will be bounded by $k+k\cdot n=k\cdot(n+1)$. According to \cite[Section~3.2]{pudlak-length-98} the contribution of the variables is at most loga\-rithmic and can in fact be avoided.
\end{proof}

Now we can deduce the following effective version of Gentzen's result:

\begin{proposition}
For each number $n$ we have
\begin{equation*}
 \pa[R]\vdash\ti(\omega_n,\widehat\gamma\,R\gamma),
\end{equation*}
with polynomial in $n$ proofs.
\end{proposition}
To be precise, let us remark that $\omega_n=\omega_n^1$ refers to the defined function symbol $(\alpha,x)\mapsto\omega^\alpha_x$ rather than the numerical code of the actual ordinal $\omega_n$. Eliminating the defined function symbol, the proposition promises polynomial proofs of the formulas $\exists_\alpha(\omega^1_n=\alpha\land\ti(\alpha,\widehat\gamma\,R\gamma))$.
\begin{proof}
In the proof from Lemma~\ref{lem:jump-transfinite-induction}, replace $J[R]$ by the equivalent formula $J'[R]$. Also instantiate $\alpha$ to $\omega^\alpha_x$ and observe $\omega^{\omega^\alpha_x}=\omega^\alpha_{x+1}$, to get a proof
\begin{equation*}
 \pa[R]\vdash\forall_x(\ti(\omega^\alpha_x,\widehat\gamma\,J'[R])\rightarrow \ti(\omega^\alpha_{x+1},\widehat\gamma\,R\gamma)).
\end{equation*}
Replace any occurrence of $Rt$ in this proof by $J'_k[R](t)$, renaming bound variables as necessary. As the axioms of $\pa[R]$ are closed under substitution this yields proofs
\begin{equation*}
 \pa[R]\vdash\forall_x(\ti(\omega^\alpha_x,\widehat\gamma\,J'_{k+1}[R])\rightarrow \ti(\omega^\alpha_{x+1},\widehat\gamma\,J'_k[R])).
\end{equation*}
In view of the previous lemma the length of these proofs is polynomial in $k$. Now instantiate $x$ to the numeral $n-(k+1)$, for each $k<n$. As $\overline{(n-(k+1))}+1=\overline{n-k}$ can be verified with polynomial proofs, we get
\begin{equation*}
 \pa[R]\vdash\ti(\omega^\alpha_{n-(k+1)},\widehat\gamma\,J'_{k+1}[R])\rightarrow \ti(\omega^\alpha_{n-k},\widehat\gamma\,J'_k[R])
\end{equation*}
with proofs of length at most $p(n,k)$, for some polynomial $p$. Combining these implications from $k=n-1$ (hence $n-(k+1)=0$) to $k=0$ (hence $J'_k[R]\equiv R\gamma$) yields a proof
\begin{equation*}
 \pa[R]\vdash\ti(\omega^\alpha_0,\widehat\gamma\,J'_n[R])\rightarrow \ti(\omega^\alpha_n,\widehat\gamma\,R\gamma).
\end{equation*}
Its length is essentially bounded by $\sum_{k<n}p(n,k)$ and thus polynomial in $n$. For the ordinal $\alpha=1$ (hence $\omega^\alpha_0=1$) we have a trivial proof of $\ti(\omega^1_0,\widehat\gamma\, R\gamma)$, as in~\cite[Lemma~4.2]{sommer95}. Substituting $J'_n[R](t)$ for $Rt$ we get polynomial proofs of the formulas $\ti(\omega^1_0,\widehat\gamma\,J'_n[R])$. A final application of modus ponens yields the desired proofs $\pa[R]\vdash\ti(\omega^1_n,\widehat\gamma\,R\gamma)$ of polynomial length.
\end{proof}
In \cite{freund-proof-length} it has been shown that any proof of $\feps(n)\!\downarrow$ in the fragment $\isigma_n$ must be extremely long. This forms an interesting contrast with the feasible proofs available in full Peano arithmetic:

\begin{corollary}
For each number $n$ we have
\begin{equation*}
\pa\vdash\feps(n)\!\downarrow,
\end{equation*}
with polynomial in $n$ proofs.
\end{corollary}
\begin{proof}
 Substitute $F_t\!\downarrow$ for $Rt$ in the proofs from the previous proposition, to get
\begin{equation*}
 \pa\vdash\ti(\omega_{n+1},\widehat\gamma\,F_\gamma\!\downarrow)
\end{equation*}
with polynomial in $n$ proofs. Instantiating Lemma~\ref{lem:ti-to-feps} to $n$ gives proofs
\begin{equation*}
 \pa\vdash\ti(\omega_{n+1},\widehat\gamma\,F_\gamma\!\downarrow)\rightarrow\feps(n)\!\downarrow
\end{equation*}
of length polynomial (in fact logarithmic) in $n$. Modus ponens yields the result.
\end{proof}

Putting things together we obtain the following main result:

\begin{theorem}
For each number $n$ we have
\begin{equation*}
\pa\vdash\con(\pa+\con^*(\pa))\!\restriction\!n,
\end{equation*}
with polynomial in $n$ proofs. The same holds with the slow reflection statement $\rfn^*_{\Pi_2}(\pa)$ at the place of the slow consistency statement $\con^*(\pa)$.
\end{theorem}
\begin{proof}
Instantiating Theorem~\ref{thm:finite-consistency-conditional} to $n$ we get proofs
\begin{equation*}
 \pa\vdash\feps(n)\!\downarrow\rightarrow\con(\pa+\con^*(\pa))\!\restriction\!n
\end{equation*}
of length polynomial (in fact logarithmic) in $n$. Note that Theorem~\ref{thm:finite-consistency-conditional} provides the same with $\rfn^*_{\Pi_2}(\pa)$ at the place of $\con^*(\pa)$. We can conclude by the previous corollary.
\end{proof}

Pudl\'ak's original conjecture states that Peano arithmetic has no polynomial proofs of $\con(\pa+\con(\pa))\!\restriction\!n$, with usual rather than slow consistency. The following shows that Theorem~\ref{thm:finite-consistency-conditional} does not hold for usual consistency. This means that Pudl\'ak's conjecture cannot be refuted by our approach:

\begin{proposition}\label{prop:usual-consistency-false}
 We have
\begin{equation*}
 \pa\nvdash\forall_x(\feps(x)\!\downarrow\,\rightarrow\con(\pa+\con(\pa))\!\restriction\!x).
\end{equation*}
\end{proposition}
\begin{proof}
 Aiming at a contradiction, assume that the given statement is provable. Then this fact can be verified in $\pa$, i.e.~we have
\begin{equation*}
 \pa\vdash\operatorname{Pr}_\pa(\forall_x(\feps(x)\!\downarrow\,\rightarrow\con(\pa+\con(\pa))\!\restriction\!x)).
\end{equation*}
As in the proof of Corollary~\ref{cor:slow-doesnt-imply-usual} we have $\pa\vdash\forall_x\operatorname{Pr}_\pa(\feps(\dot x)\!\downarrow)$. Thus we get
\begin{equation*}
 \pa\vdash\forall_x\operatorname{Pr}_\pa(\con(\pa+\con(\pa))\!\restriction\!\dot x).
\end{equation*}
Since consistency implies $\Pi_1$-reflection this yields
\begin{equation*}
 \pa+\con(\pa)\vdash\forall_x\con(\pa+\con(\pa))\!\restriction\!x,
\end{equation*}
which contradicts G\"odel's second incompleteness theorem.
\end{proof}

In view of this negative result, let us conclude with the following general question:

\begin{question}For any r.e.\ theory $\mathbf T$ one could ask about the least recursive ordinal~$\alpha$ (in some natural ordinal notation system)  such that
\begin{equation*}
\isigma_1\vdash\forall_x(F_\alpha(x)\!\downarrow\rightarrow\con(\mathbf T)\!\restriction\!x).
\end{equation*}
In particular, we find the question interesting for $\isigma_n$ ($n\ge 2$), $\pa+\con(\pa)$, and natural fragments of second-order arithmetic.  How does it relate to the proof-theoretic ordinal of the theory $\mathbf T$? 
\end{question}

\bibliographystyle{amsplain} 
\bibliography{Short-Proofs-Slow-Consistency}

\end{document}